\documentclass[11pt]{amsart}
\usepackage{amsmath, amsthm, amssymb, latexsym,url}
\usepackage[mathscr]{eucal}
\usepackage{hyperref}

\author[D. Airey]{Dylan Airey}
\address[D. Airey]{
Department of Mathematics, University of Texas at Austin, 2515 Speedway, Austin, TX 78712-1202, USA}
\email{dylan.airey@utexas.edu}

\author[B. Mance]{Bill Mance}
\address[B. Mance]{Department of Mathematics, University of North Texas, General Academics Building 435, 1155 Union Circle,  \#311430, Denton, TX 76203-5017, USA}
\email{mance@unt.edu}

\author[J. Vandehey]{Joseph Vandehey}
\address[J. Vandehey]{Department of Mathematics, University of Georgia at Athens, Boyd graduate studies research center, Athens, GA 30606 USA}
\email{vandehey@uga.edu}

\thanks{Research of the first and second authors is partially supported by the U.S. NSF grant DMS-0943870.  We would like to thank Samuel Roth for posing the problem that led to \reft{main} and \reft{computability} to the second author  at the 2012 RTG conference: Logic, Dynamics and Their Interactions, with a Celebration of the Work of Dan Mauldin in Denton, Texas.  He asked if it is true that $nx \in \NQ$ for all natural numbers $n$ implies that $x \in \DNQ$.}

\title[Normality preserving operations part II]{Normality preserving operations for Cantor series expansions and associated fractals part II}
\date{\today}

\newtheorem{thm}{Theorem}[section]
\newtheorem{cor}[thm]{Corollary}

\newtheorem{lem}[thm]{Lemma}

\newtheorem{problem}[thm]{Problem}

\newtheorem{definition}[thm]{Definition}

\newcommand{\labeq}[1]{\label{eq:#1}}
\newcommand{\refeq}[1]{(\ref{eq:#1})}
\newcommand{\labt}[1]{\label{thm:#1}}
\newcommand{\reft}[1]{Theorem~\ref{thm:#1}}

\newcommand{\refl}[1]{Lemma~\ref{lemma:#1}}
\newcommand{\labd}[1]{\label{definition:#1}}

\newcommand{\dimh}[1]{\hbox{$\dim_{\hbox{H}}$}\left( #1\right)}

\newcommand{\NN}{\mathbb{N}_2^{\mathbb{N}}}

\newcommand{\wrt}[1]{\hbox{ w.r.t. }#1}
\newcommand{\wrtQ}{\hbox{ w.r.t. }Q}

\newcommand{\floor}[1]{\left\lfloor #1 \right\rfloor} 
 
\newcommand{\br}[1]{\left\{ #1 \right\}}

\newcommand{\pr}[1]{\left( #1 \right)}

\newcommand{\NQ}{\mathscr{N}(Q)}
\newcommand{\N}[1]{\mathscr{N}( #1 )}
\newcommand{\Nk}[2]{\mathscr{N}_{#2}( #1 )} 
\newcommand{\DNQ}{\mathscr{DN}(Q)}
 
\newcommand{\RNQ}{\mathscr{RN}(Q)}

\newcommand{\RDN}{\RNQ \cap \DNQ \backslash \NQ}

\newcommand{\taursx}{\tau_{r,s}(x)}
\newcommand{\urset}{\Xi(Q)}

\allowdisplaybreaks

\begin{document}

\maketitle

\begin{abstract}
We investigate how non-zero rational multiplication and rational addition affect normality with respect to $Q$-Cantor series expansions. In particular, we show that there exists a $Q$ such that the set of real numbers which are $Q$-normal but not $Q$-distribution normal, and which still have this property when multiplied and added by rational numbers has full Hausdorff dimension. Moreover, we give such a number that is explicit in the sense that it is computable. 
\end{abstract}

\section{Introduction}
Let $\N{b}$ be the set of numbers normal in base $b$ and let $f$ be a  function from $\mathbb{R}$ to $\mathbb{R}$.  We say that $f$ {\it preserves $b$-normality} if $f(\N{b}) \subseteq \N{b}$.  We can make a similar definition for preserving normality with respect to continued fraction expansions, $\beta$-expansions, the L\"uroth series expansion, etc.

Several authors have studied $b$-normality preserving functions.  Some $b$-normality preserving functions naturally arise in  H. Furstenberg's work on disjointness in ergodic theory\cite{FurstenbergDisjoint}.  V. N. Agafonov \cite{AgafonovNormal}, T. Kamae \cite{Kamae}, T. Kamae and B. Weiss \cite{KamaeWeiss}, and W. Merkle and J. Reimann \cite{MerkleReimann} studied $b$-normality preserving selection rules.

For a real number $r$, define real functions $\pi_r$ and $\sigma_r$ by $\pi_r(x)=rx$ and $\sigma_r(x)=r+x$.  In 1949 D. D. Wall proved in his Ph.D. thesis \cite{Wall} that for non-zero rational $r$ the function $\pi_r$ is $b$-normality preserving for all $b$ and that the function $\sigma_r$ is $b$-normality preserving functions for all $b$ whenever $r$ is rational.  These results were also independently proven by K. T. Chang in 1976 \cite{ChangNormal}.  D. D. Wall's method relies on the well known characterization that a real number $x$ is normal in base $b$ if and only if the sequence $(b^nx)$ is uniformly distributed mod $1$ that he also proved in his Ph.D. thesis.

D. Doty, J. H. Lutz, and S. Nandakumar took a substantially different approach from D. D. Wall and strengthened his result.  They proved in \cite{LutzNormalityPreserves} that for every real number $x$ and every non-zero rational number $r$  the $b$-ary expansions of $x, \pi_r(x),$ and $\sigma_r(x)$ all have the same finite-state dimension and the same finite-state strong dimension.  It follows that $\pi_r$ and $\sigma_r$ preserve $b$-normality.  It should be noted that their proof uses different methods from those used by D. D. Wall and is unlikely to be proven using similar machinery.

C. Aistleitner generalized D. D. Wall's result on $\sigma_r$.  Suppose that $q$ is a rational number and that the digits of the $b$-ary expansion of $z$ are non-zero on a set of indices of density zero.  In \cite{AistleitnerNormalPreserves} he proved that the function $\sigma_{qz}$ is $b$-normality preserving.  
It was shown in \cite{AireyManceNormalPreserves} that  C. Aistleitner's result does not generalize to at least one notion of normality for some of the Cantor series expansions.
%Ainstler result extending Wall in uniform dist journal.  Say the few things he did (look it up)  \cite{AistleitnerNormalPreserves}.  Also add Volkmann and Maxfield references.  Maybe above for Maxfield.  Also did Ainstler improve on Pellegrino or how should that be included?

There are still many open questions relating to the functions $\pi_r$ and $\sigma_r$.  For example, M. Mend\'{e}s France asked in \cite{X} if the function $\pi_r$ preserves simple normality with respect to the regular continued fraction for every non-zero rational $r$.  The authors are unaware of any theorems that state that either $\pi_r$ or $\sigma_r$ preserve any other form of normality than $b$-normality.

In this paper we will be interested in the function $\tau_{r,s}=\sigma_s \circ \pi_r$ for $r\in \mathbb{Q}\backslash \{0\}$ and $s\in \mathbb{Q}$, and how this function preserves certain notions of normality of $Q$-Cantor series expansions, namely $Q$-normality and $Q$-distribution normality. (We will provide definitions for all these terms in Section \ref{sec:Cantor}.) In Theorem \ref{thm:main}, we will show that there exists a basic sequence $Q$ and a real number $x$ such that $\taursx$ is always $Q$-normal and always \emph{not} $Q$-distribution normal; in fact, we will show that for this $Q$, the set of $x$ with this property is big in the sense that it has full Hausdorff dimension. 
It was first shown in \cite{AlMa} that the set of numbers that are $Q$-normal but not $Q$-distribution normal is non-empty for some basic sequences $Q$, but no indication was given to the size of this set.  
For a specific basic sequence $Q$, we show that there exists a subset $\urset$  of the set of $Q$-normal numbers that is invariant under $\tau_{r,s}$ for every $r\in \mathbb{Q}\backslash \{0\}$ and $s\in \mathbb{Q}$ (i.e.  $\tau_{r,s}\left(\urset\right)=\urset$) and  has full Hausdorff dimension.  
Related questions for the Cantor series expansions are studied in \cite{AireyManceNormalPreserves}.

It is an interesting question to know how explicit this $x$ and $Q$ are, so  we bring in some definitions from recursion theory.  A real number $x$ is \textit{computable} if there exists $b \in \mathbb{N}$ with $b\geq 2$ and a total recursive function $f: \mathbb{N} \to \mathbb{N}$ that calculates the digits of $x$ in base $b$. A sequence of real numbers $(x_n)$ is \textit{computable} if there exists a total recursive function $f: \mathbb{N}^2 \to \mathbb{Z}$ such that for all $m,n$ we have that $\frac{f(m,n)-1}{m} < x_n< \frac{f(m,n)-1}{m}$.

%A real number is \textit{absolutely normal} if it is normal in base $b$ for all bases $b$.
M. W. Sierpi\'{n}ski gave an example of an absolutely normal number that is not computable in \cite{Sierpinski}.
The authors feel that examples such as M. W. Sierpi\'{n}ski's are not fully explicit since they are not computable real numbers, unlike Champernowne's number. 
A. M. Turing gave the first example of a computable absolutely normal number in an unpublished manuscript.  This paper may be found in his collected works \cite{Turing}. See \cite{BecherFigueiraPicchi} by V. Becher, S. Figueira, and R. Picchi for further discussion.
%\footnote{The $n$'th digit of A. M. Turing's number may be computed with an algorithm that is doubly exponential in $n$. V. Becher, P. A. Heiber, and T. A. Slaman constructed an absolutely normal number in \cite{BecherHeiberSlaman} whose digits may be computed in polynomial time.}. 
In Theorem \ref{thm:computability} we give a basic sequence $Q$ and real number $x$, with $x$ in the set discussed in \reft{main}, that are fully explicit in the sense that they are computable as a sequence of integers and a real number, respectively.

%related question involving fractals studied in vandehey paper  $\N{b}$

Throughout this paper we will use a number of standard asymptotic notations. By $f(x)=O(g(x))$ we mean that there exists some real number $C>0$ such that $|f(x)| \le C |g(x)|$. By $f(x) \asymp g(x)$, we mean $f(x)=O(g(x))$ and $g(x)=O(f(x))$. By $f(x)=o(g(x))$, we mean that $f(x)/g(x) \to 0$ as $x\to \infty$.

\section{Cantor series expansions}\label{sec:Cantor}

The study of normal numbers and other statistical properties of real numbers with respect to large classes of Cantor series expansions was  first done by P. Erd\H{o}s and A. R\'{e}nyi in \cite{ErdosRenyiConvergent} and \cite{ErdosRenyiFurther} and by A. R\'{e}nyi in \cite{RenyiProbability}, \cite{Renyi}, and \cite{RenyiSurvey} and by P. Tur\'{a}n in \cite{Turan}.

%\cite{Renyi,RenyiSurvey,ErdosRenyiConvergent,ErdosRenyiFurther,RenyiProbability}
The $Q$-Cantor series expansions, first studied by G. Cantor in \cite{Cantor},
are a natural generalization of the $b$-ary expansions.\footnote{G. Cantor's motivation to study the Cantor series expansions was to extend the well known proof of the irrationality of the number $e=\sum 1/n!$ to a larger class of numbers.  Results along these lines may be found in the monograph of J. Galambos \cite{Galambos}. } %   See also \cite{TijdemanYuan} and \cite{HT}.  }
Let $\mathbb{N}_k:=\mathbb{Z} \cap [k,\infty)$.  If $Q \in \NN$, then we say that $Q$ is a {\it basic sequence}.
% if each $q_n$ is an integer greater than or equal to $2$.
Given a basic sequence $Q=(q_n)_{n=1}^{\infty}$, the {\it $Q$-Cantor series expansion} of a real number $x$  is the (unique)\footnote{Uniqueness can be proven in the same way as for the $b$-ary expansions.} expansion of the form
\begin{equation} \labeq{cseries}
x=E_0+\sum_{n=1}^{\infty} \frac {E_n} {q_1 q_2 \cdots q_n}
\end{equation}
where $E_0=\floor{x}$ and $E_n$ is in $\{0,1,\ldots,q_n-1\}$ for $n\geq 1$ with $E_n \neq q_n-1$ infinitely often. We abbreviate \refeq{cseries} with the notation $x=E_0.E_1E_2E_3\cdots$ w.r.t. $Q$.

A {\it block} is an ordered tuple of non-negative integers, a {\it block of length $k$} is an ordered $k$-tuple of integers, and {\it block of length $k$ in base $b$} is an ordered $k$-tuple of integers in $\{0,1,\ldots,b-1\}$.

Let
$$
Q_n^{(k)}:=\sum_{j=1}^n \frac {1} {q_j q_{j+1} \cdots q_{j+k-1}} \hbox{ and }  T_{Q,n}(x):=\left(\prod_{j=1}^n q_j\right) x \pmod{1}.
$$
A. R\'enyi \cite{Renyi} defined a real number $x$ to be {\it normal} with respect to $Q$ if for all blocks $B$ of length $1$,
\begin{equation}\labeq{rnormal}
\lim_{n \rightarrow \infty} \frac {N_n^Q (B,x)} {Q_n^{(1)}}=1,
\end{equation}
where $N_n^Q(B,x)$ is the number of occurences of the block $B$ in the sequence $(E_i)_{i=1}^n$ of the first $n$ digits in the $Q$-Cantor series expansion of $x$. If $q_n=b$ for all $n$ and we restrict $B$ to consist of only digits less than $b$, then \refeq{rnormal} is equivalent to {\it simple normality in base $b$}, but not equivalent to {\it normality in base $b$}. 
A basic sequence $Q$ is {\it $k$-divergent} if
$\lim_{n \rightarrow \infty} Q_n^{(k)}=\infty$ and {\it fully divergent} if $Q$ is $k$-divergent for all $k$. % and {\it $k$-convergent} if it is not $k$-divergent.  
A basic sequence $Q$ is {\it infinite in limit} if $q_n \rightarrow \infty$.

\begin{definition}\labd{1.7} A real number $x$  is {\it $Q$-normal of order $k$} if for all blocks $B$ of length $k$,
$$
\lim_{n \rightarrow \infty} \frac {N_n^Q (B,x)} {Q_n^{(k)}}=1.
$$
We let $\Nk{Q}{k}$ be the set of numbers that are $Q$-normal of order $k$.  The real number $x$ is {\it $Q$-normal} if
$x \in \NQ := \bigcap_{k=1}^{\infty} \Nk{Q}{k}.$
A real number~$x$ is {\it $Q$-distribution normal} if
the sequence $(T_{Q,n}(x))_{n=0}^\infty$ is uniformly distributed mod $1$.  Let $\DNQ$ be the set of $Q$-distribution normal numbers.
\end{definition}

It  follows from a well known result of H. Weyl \cite{Weyl2,Weyl4} that $\DNQ$ is a set of full Lebesgue measure for every basic sequence $Q$. We will need the following results of the second author \cite{Mance4} later in this paper.

\begin{thm}\labt{measure}\footnote{Early work in this direction has been done by A. R\'enyi \cite{Renyi}, T.  \u{S}al\'at \cite{Salat4}, and F. Schweiger \cite{SchweigerCantor}.}
Suppose that $Q$ is infinite in limit.  Then $\Nk{Q}{k}$ (resp. $\NQ$) is of full measure if and only if $Q$ is $k$-divergent (resp. fully divergent). 
\end{thm}

We note the following simple theorem.

\begin{thm}\label{thm:distributionnormal}
Suppose that $Q$ is infinite in limit. Then $x=E_0.E_1E_2\dots$ is $Q$-distribution normal if and only if the sequence $(E_n/q_n)_{n=1}^\infty$ is uniformly distributed modulo 1.
\end{thm}

Note that in base~$b$, where $q_n=b$ for all $n$,
 the corresponding notions of $Q$-normality and $Q$-distribution normality are equivalent. This equivalence
is fundamental in the study of normality in base $b$. 

Another definition of normality, $Q$-ratio normality, has also been studied.  We do not introduce this notion here as this set contains the set of $Q$-normal numbers and all results in this paper that hold for $Q$-normal numbers also hold for $Q$-ratio normal numbers.
The complete containment relation between the sets of these normal numbers and pair-wise intersections thereof is proven in \cite{ppq1}.  The Hausdorff dimensions of difference sets such as $\RDN$ are computed in \cite{AireyManceHDDifference}.
%A surprising property of $Q$-normality of order $k$ is that we may not conclude that $\Nk{Q}{k} \subseteq \Nk{Q}{j}$ for all $j <k$ like we may for the $b$-ary expansions.  In fact, it was shown in \cite{ppq2} that for every $k$ there exists a basic sequence $Q$ and a real number $x$ such that $\Nk{Q}{k} \backslash \bigcup_{j=1}^{k-1} \Nk{Q}{j}$ is non-empty.  Thus, rather than showing that some functions do not preserve $Q$-normality of order $k$, we will show that they do not preserve $Q$-normality of any order.  We will always demonstrate numbers not $Q$-normal of any order that either have at most finitely many copies of the digit $0$ or the digit $1$ in their $Q$-Cantor series expansion.
Set
$$
\urset=\br{x=0.E_1E_2\cdots\wrt{Q} : \taursx \in \NQ \backslash \DNQ \ \forall r \in \mathbb{Q}\backslash \{0\}, s \in \mathbb{Q} }.
$$
Our main results of this paper will be the following:

\begin{thm}\label{thm:main}
There exists a basic sequence $Q$ such that the Hausdorff dimension of $\urset$ is $1$.
\end{thm}

\begin{thm}\label{thm:computability}
There exists a computable basic sequence $Q$ and a computable real number $x$ in $\urset$.
% such that $\taursx$ is in $ \NQ \backslash \DNQ $ for all $r \in \mathbb{Q}\backslash\{0\}$ and $s\in \mathbb{Q}$, \emph{and} both $Q$ and $x$ are computable.
\end{thm}

\subsection{The digits of $\taursx$}

In order to prove the main results of this paper, we will want to understand how the digits of $\taursx$ differ from the digits of $x$, when $x$ takes a specific form. We begin with some lemmas based on elementary calculations.

\begin{lem}
If $x=p/q$ is a rational number with $p\in \mathbb{Z}$, $q\in \mathbb{N}$ and $q\mid q_1q_2\dots q_N$ for some $N$, then $x$ has a finite $Q$-Cantor series expansion of the form
\[
x= E_0+\sum_{n=1}^N \frac{E_n}{q_1q_2\dots q_n}.
\]
Alternately if $x$ is a real number in the interval $[0,1/q_1q_2\dots q_N)$, then $x$ has a $Q$-Cantor series expansion of the following form,
\[
x=\sum_{n=N+1}^\infty \frac{E_n}{q_1q_2\dots q_n}
\]
so that $E_n=0$ for $n\le N$.
\end{lem}

This allows us to prove a number of additional lemmas rather trivially.

\begin{lem}
Suppose that $x=E_0.E_1E_2\cdots\wrtQ$.
%\[
%x= E_0+\sum_{n=1}^\infty \frac{E_n}{q_1q_2\dots q_n}.
%\]
If $s=p/q$ is rational with $p\in \mathbb{Z}$, $q\in \mathbb{N}$ and $q\mid q_1 q_2 \dots q_N$, then $\sigma_s(x)$ has a $Q$-Cantor series expansion of the form
\[
\sigma_s(x)= E'_0 + \sum_{n=1}^N \frac{E'_n}{q_1q_2\dots q_n} + \sum_{n=N+1}^\infty \frac{E_n}{q_1q_2 \dots q_n}
\]
so that $\sigma_s(x)$ and $x$ differ only in their first $N+1$ digits.
\end{lem}

\begin{cor}\label{lem:rationaladd}
Suppose that $Q$ has the property that for any integer $n$ there exists an integer $m$ such that $n | q_m$. Then for any rational number $s$, the $Q$-Cantor series expansion of $x$ and of $\sigma_s(x)$ differ on at most finitely many places.
\end{cor}

\begin{lem}\label{lem:rationalmult}
Suppose that $x$ has a finite $Q$-Cantor series expansion of the form
\[
x= \sum_{n=N}^M \frac{E_n}{q_1q_2\dots q_n}.
\]
We write 
\[
E=  E_N q_N q_{N+1}\dots q_M + E_{N+1} q_{N+1} q_{N+2}\dots q_{M} + \dots +E_{M-1} q_{M-1} +E_M
\]
\[
q = q_Nq_{N+1}\dots q_M
\]
so that
\[
x= \frac{E}{q_1q_2\dots q_{N-1} q}.
\]

Suppose $r$ is a nonzero rational number. If $rE$ is an integer and $rE<q$, then $\pi_r(x)$ has a finite $Q$-Cantor series expansion of the form
\[
\pi_r(x)= \sum_{n=N}^M \frac{E'_n}{q_1q_2\dots q_n}.
\]
\end{lem}

\section{Results on Hausdorff dimension}\label{sec:Hausdorff}

Given basic sequences $\alpha = (\alpha_i)$ and  $\beta = (\beta_i)$, 
sequences of non-negative integers $s = (s_i), t = (t_i), \upsilon = (\upsilon_i),$ and $F = (F_i)$, and a sequence of sets $I = (I_i)$ such that $I_i \subseteq \{0, 1, \cdots, \beta_i-1\}$, define the set $\Theta(\alpha, \beta, s, t, \upsilon, F, I)$ as follows.
Let $Q=Q(\alpha, \beta, s, t, \upsilon)=(q_n)$ be the following basic sequence:
\begin{equation}\label{eq:Moransetbase}
\left[[\alpha_1]^{s_1} [\beta_1]^{t_1}\right]^{\upsilon_1} \left[[\alpha_2]^{s_2} [\beta_2]^{t_2}\right]^{\upsilon_2} \left[[\alpha_3]^{s_3} [\beta_3]^{t_3}\right]^{\upsilon_3} \cdots.
\end{equation}
Define the function
$$
i(n) = \min \br{t : \sum_{i=1}^{t-1} \upsilon_i (s_i + t_i) < n}.
$$
Set $\Phi_\alpha(i,c,d) = \sum_{j =1}^{i-1} \upsilon_j s_j + c s_i + d$ where $0 \leq c < \upsilon_i$ and $0 \leq d < s_i$ and let the functions $i_\alpha(n)$, $c_\alpha(n)$, and $d_\alpha(n)$ be such that $\Phi^{-1}_\alpha(n) = (i_\alpha(n), c_\alpha(n), d_\alpha(n))$. Note this is possible since $\Phi_\alpha$ is a bijection from $\mathcal{U} = \br{(i,c,d) \in \mathbb{N}^3 : 0 \leq c < \upsilon_i, 0 \leq d < s_i}$ to $\mathbb{N}$. Define the function
$$
G(n) = \sum_{j=1}^{i_\alpha(n)-1} \upsilon_j (s_j+t_j) + c_\alpha(n)\pr{s_{i_\alpha(n)} + t_{i_\alpha(n)}} + d_\alpha(n).
$$ 

We consider the condition on $n$
\begin{equation}\labeq{VNcond}
\pr{n-\sum_{j=1}^{i(n)-1} \upsilon_j(s_j+t_j)} \mod (s_{i(n)}+t_{i(n)}) \geq s_{i(n)}.
\end{equation}
Define the intervals 
$$
V(n) =
\begin{cases}
	I_{i(n)} & \text{ if condition \refeq{VNcond} holds}  \\
\ \\
	\left [F_{G(n)}, F_{G(n)}+1 \right ) & \text{ else}
\end{cases}.
$$
That is, we choose digits from $I_{i(n)}$ in positions corresponding to the bases obtained from the sequence $\beta$ and choose a specific digit from $F$ for the bases obtained from the sequence $\alpha$.
Set
$$
\Theta(\alpha, \beta, s, t, \upsilon, F, I) = \br{x = 0.E_1 E_2 \cdots \wrt{Q} : E_n \in V(n)}.
$$
We will need the following lemma from \cite{AireyManceHDDifference}.
\begin{lem}\label{HDT}
Suppose that basic sequences $\alpha = (\alpha_i)$ and  $\beta = (\beta_i)$, 
sequences of non-zero integers  $s = (s_i), t = (t_i), \upsilon = (\upsilon_i),$ and $F = (F_i)$, and a sequence of sets $I = (I_i)$ such that $I_i \subseteq \{0, 1, \cdots, \beta_i -1\}$ are given where $\lim_{n \to \infty} |I_i| = \infty$ and %the following conditions hold:
$$
%\labeq{HDT1}
\lim_{n \to \infty} \frac{s_n \log \alpha_n}{\sum_{i=1}^{n-1} \upsilon_i t_i \log \beta_i} = \lim_{n \to \infty} \frac{s_n \log \alpha_n}{t_n \log \beta_n} = 0.
%\labeq{HDT2}
%\lim_{n \to \infty} \frac{s_n \log \alpha_n}{t_n \log \beta_n} &= 0;\\
%\labeq{HDT3}
%\lim_{n \to \infty} \frac{\log |I_i|}{\log \beta_n} &= \gamma; \\
%\labeq{HDT4}
%&\lim_{n \to \infty} |I_i| = \infty.
$$
Then $\dimh{\Theta(\alpha, \beta, s, t, \upsilon, F, I)} = \lim_{n \to \infty} \frac{\log |I_i|}{\log \beta_n}$ provided this limit exists.
\end{lem}

\iffalse
\begin{thm}
Let $Q$ be the base Vandehey constructed.  Then
$$
\dimh{\urset}=1.
$$
\end{thm}

\begin{proof}
Let $\xi=0.E_1E_2\cdots\wrt{Q}$ be (ONE OF THE NUMBERS CONSTRUCTED BY VANDEHEY IN SOME SECTION). Let $\alpha_i = i$, $\beta_i = (i!)^2$, $s_i = t_i = n_i$, $\upsilon_i = L_i l_i$, 
$$
I_i = \br{ i, i+1, \cdots, \beta_i^{1-\log(i)^{-\frac{1}{2}}} } \cap i! \mathbb{Z},
$$
and $F_i = E_{G(i)}$.  By \refl{HDT}, we know that $\dimh{\Theta(\alpha, \beta, s, t, \upsilon, F, I)} = 1$.  But by (Vandehey's main theorem), we know that 
$$
\Theta(\alpha, \beta, s, t, \upsilon, F, I) \subseteq \urset,
$$
so
$
\dimh{\urset}=1.
$
\end{proof}

\fi

%%%%%%%%%%%%%%%%%%%%%%%
\section{Lemmas on $(\epsilon,k)$-normal sequences}
%%%%%%%%%%%%%%%%%%%%%%%

Given integers $b\ge 2, n\ge 1, k \ge 1$, let $p_b(n,k)$ denote the number of blocks of length $n$ in base $b$ containing exactly $k$ copies of a given digit. (By symmetry  it does not matter which digit we are interested in.)  

\begin{lem}[Lemma 4.7 in \cite{BugeaudBook}]\label{lem:bugeaud}
Let $b\ge 2$ and $n \ge b^{15}$ be integers. For every real number $\epsilon$ with $n^{-1/3} \le \epsilon \le 1$, we have
\[
\sum_{-n\le j \le -\lceil \epsilon n \rceil} p_b(bn,n+j) + \sum_{\lceil \epsilon n \rceil \le j \le (b-1)n} p_b(bn,n+j) \le 2^{14} b^{bn} e^{-\epsilon^2 n/(10 b)}.
\]
\end{lem}

\begin{lem}\label{lem:k1}
Let $b\ge 2$ and $n \ge b^{16}$ be integers. For every real number $\epsilon$ with $n^{-1/3} \le \epsilon \le 2/b$, we have
\[
\left( \sum_{j > (b^{-1}+\epsilon)n} + \sum_{j < (b^{-1}-\epsilon)n} \right) p_b(n,j) \le 2^{14} b^n e^{-\epsilon^2 n/80}.
\]
\end{lem}

\begin{proof}
Note that $p_b(n,j)$ is increasing as a function of $n$, therefore
\[
\left( \sum_{j > (b^{-1}+\epsilon)n} + \sum_{j <(b^{-1}-\epsilon)n} \right) p_b(n,j) \le \left( \sum_{j > (b^{-1}+\epsilon)n} + \sum_{j < (b^{-1}-\epsilon)n} \right) p_b(b\lfloor n/b \rfloor ,j) .
\]

Now let $\epsilon'=b\epsilon/2$ and note that
\begin{align*}
\left\lfloor \frac{n}{b} \right\rfloor+ \left\lceil \epsilon' \left\lfloor \frac{n}{b} \right\rfloor \right\rceil &\le \frac{n}{b}+\epsilon' \frac{n}{b}+1\\
&= (b^{-1}+\epsilon)n + \left( 1 - \frac{n\epsilon}{2}\right)\\ &\le (b^{-1} + \epsilon)n.
\end{align*}
Likewise one can show that
\[
\left\lfloor \frac{n}{b} \right\rfloor- \left\lceil \epsilon' \left\lfloor \frac{n}{b} \right\rfloor \right\rceil  \ge (b^{-1}-\epsilon)n.
\]
As a result, we have that 
\begin{align*}
& \left( \sum_{j > (b^{-1}+\epsilon)n} + \sum_{j < (b^{-1}-\epsilon)n} \right) p_b(b\lfloor n/b \rfloor ,j)  \\ &\qquad\le \sum_{ j \le -\lceil \epsilon' \lfloor n/b \rfloor \rceil} p_b(b\lfloor n/b \rfloor ,\lfloor n/b \rfloor +j) + \sum_{\lceil \epsilon \lfloor n/b \rfloor  \rceil \le j } p_b(b\lfloor n/b \rfloor ,\lfloor n/b \rfloor +j).
\end{align*}

We now can apply Lemma \ref{lem:bugeaud} to see that 
\begin{align*}
\left( \sum_{j > (b^{-1}+\epsilon)n} + \sum_{j <(b^{-1}-\epsilon)n} \right) p_b(n,j)& \le 2^{14} b^{b\lfloor n/b \rfloor } e^{-{\epsilon'}^2 \lfloor n/b\rfloor /(10 b)} \\
&\le 2^{14} b^n e^{-\epsilon^2 n/80},
\end{align*}
as desired. Here we made use of the fact that $\lfloor n/b \rfloor \ge n/2b$.
\end{proof}

We will say a block $B$ of length $n$ in base $b$ is $(\epsilon,k)$-normal (with respect to $b$), if the total number of occurrences in $B$ of any subblock of length $k$ in base $b$ is between $(b^{-k}-\epsilon)n$ and $(b^{-k}+\epsilon)n$. Let $B_b(n,\epsilon,k)$ denote the number of blocks of length $n$ that are \emph{not} $(\epsilon,k)$-normal with respect to $b$. Note that Lemma \ref{lem:k1} gives a bound on $B_b(n,\epsilon,1)$. The following lemma will give a bound on  $B_b(n,\epsilon,k)$.

\begin{lem}\label{lem:epsilonk}
Suppose $b\ge 2$, $k\ge 1$, $n \ge k (b^{16k}+1)$ are integers. For every real number $\epsilon$ with $2\lfloor n/k\rfloor^{-1/3} \le \epsilon \le 2/b^k$ we have
\[
B_b(n,\epsilon,k) \le  2^{15} kb^{n+k} e^{-\epsilon^2n/(160k)}.
\]
\end{lem}

\begin{proof}
Let us begin by considering an arbitrary block $B=[d_1,d_2,\dots,d_n]$ of $n$ digits in base $b$. Suppose that $n=n' k + r$ for some $r\in \{0,1\dots,k-1\}$.

Let $D_i = d_i b^{k-1} + d_{i+1} b^{k-2}+\dots + d_{i+k}$ for $1\le i \le n-k$. Note that $D_i \in \{0,1,\dots,b^k-1\}$.
For $0 \le i < k$, let $B_i = [D_i, D_{k+i}, D_{2k+i},\dots,D_{(n'-1)k+i}]$ if $i \le r$ and $B_i =  [D_i, D_{k+i}, \\ D_{2k+i},\dots,D_{(n'-2)k+i}]$ otherwise.

By the pigeon-hole principle, if $B$ is not $(\epsilon,k)$-normal with respect to $b$, then some $B_i$ is not $(\epsilon,1)$-normal with respect to $b^k$. Thus, the total number of blocks $B$ which are not $(\epsilon,k)$-normal with respect to $b$ is at most a sum over $i$ of the number of blocks $B_i$ which are not $(\epsilon,1)$-normal with respect to $b^k$, times either $b^r$ or $b^{k+r}$ to account for all possibilities of those digits of $B$ which are not contained in $B_i$.

Thus, by Lemma \ref{lem:k1}, we have
\begin{align*}
B_b(n,\epsilon,k) &\le (r+1) b^{r} 2^{14}(b^k)^{n'}e^{-\epsilon^2 n'/80}\\ &\qquad + (k-r-1) b^{k+r} 2^{14} (b^k)^{(n'-1)}e^{-\epsilon^2( n'-1)/80}\\
&\le k 2^{14} b^{k( n'+1)+r} e^{-\epsilon^2 n'/80} (1+ e^{\epsilon^2/80})\\
&\le 2^{15} kb^{n+k} e^{-\epsilon^2n/(160k)},
\end{align*}
where here again we use that $\lfloor n/k \rfloor \ge n/2k$.
\end{proof}

%%%%%%%%%%%%%%%%%%%%%%%%%%
\section{Proof of Theorem \ref{thm:main}}\label{sec:mainproof}
%%%%%%%%%%%%%%%%%%%%%%%%%%%%

Given $i\ge 2$, consider the following definitions. We let $n_i=i^{\lfloor \log i \rfloor}$, $\epsilon_i = n_i^{-1/4}$. With these definitions, we have that the number of $(\epsilon_i,k)$-normal blocks of $n_i$ digits in base $i$ is bounded by
$
i^{n_i}e^{-n_i^{1/5}},
$
provided that $i$ is sufficiently large compared to $k$. When $i=1$, we shall let $n_i=0$.

Given a block $B=[d_1,d_2,\dots,d_{n_i}]$ of $n_i$ in base $i$, let $\overline{B}=d_1i^{n-1}+d_2 i^{n-2}+\dots+d_n$ be the naturally associated integer. Let $\mathcal{L}_i$ denote the set of all such blocks $B$ such that $i!\overline{B} < i^{n_i}$ and $i!|\overline{B}$. Note that $\mathcal{L}_i$ always contains the block $[0,0,\dots,0]$. We denote the size of $\mathcal{L}_i$ by $\ell_i$, and note that $\ell_i \asymp i^n/(i!)^2$ for sufficiently large $i$. We will let 
\[
L_i=i! \left\lceil\frac{ n_{i+1}\ell_{i+1}}{n_i\ell_i} \right\rceil.
\]

 In the Moran set construction given in section \ref{sec:Hausdorff}, let $\alpha_i = i$, $\beta_i = (i!)^2$, $s_i=t_i=n_i$, and $v_i=L_i \ell_i$, with $Q$ given by \eqref{eq:Moransetbase}. We shall also let 
\[
I_i = \left\{1,2,\dots,\left\lfloor \beta_i^{1-\log(i)^{-1}}\right\rfloor\right\}\cap \left( \left\lfloor \sqrt{i} \right\rfloor !\right)\mathbb{Z}.
\]
With this definition, we have that $\log |I_i|/\log \beta_i$ tends to $1$ and that, as $i$ grows, all elements of $I_i$ become arbitrarily small compared to $\beta_i$ and are eventually divisible by any fixed integer.
Since $n_1=0$, the smallest base in $Q$ constructed this way is $2$, so that $Q$ really is a basic sequence. 

With these definitions (and any appropriate choice of sequence $(F)$), it is easy to check that all such points satisfy the conditions of Theorem \ref{HDT}, so that
$\dimh{\Theta(\alpha, \beta, s, t, \upsilon, F, I)} = 1$.
It therefore suffices to show that for some proper selection of $F$, we have $\Theta(\alpha, \beta, s, t, \upsilon, F, I) \subset \urset$.
To make this selection of $F$, let 
\[
X_i = \left[[i]^{n_i}[(i!)^2]^{n_i}\right]^{\ell_i},
\]
so that we could alternately write $Q$ as 
\begin{equation}\label{eq:Qconstruction}
Q=[X_2]^{L_2} [X_3]^{L_3} [X_4]^{L_4}\cdots.
\end{equation}
We shall then choose the digits of $F$ in such a way so that the digits corresponding to the the $j$th occurence of the bases $[i]^{n_i}$ in each copy of $X_i$ are the $j$th string from $\mathcal{L}_i$ (when ordered lexicographically).

With this definition of $F$ in mind, let $x$ be any point in $\Theta(\alpha, \beta, s, t, \upsilon, F, I)$, $r\in \mathbb{Q}\backslash\{0\}$, and $s\in \mathbb{Q}$. We will show that $\taursx$ is $Q$-normal but not $Q$-distribution normal. By the construction of $Q$ and Corollary \ref{lem:rationaladd}, we have for any rational number $s$ that the $Q$-Cantor series expansions of $\taursx$ and $\pi_r(x)$ differ on at most finitely many digits. In addition, we have that $\overline{B}$ for $B\in \mathcal{L}_i$ is small compared with $i^{n_i}$ and is divisible by $i!$, and each digit of $I_i$ is small compared with $(i!)^2$ and is divisible by $\lfloor \sqrt{i}\rfloor !$. Therefore, by Lemma \ref{lem:rationalmult}, we have that for any nonzero rational number $r$, there will be a sufficiently large $i$ such that the digits of $\taursx$ corresponding to the bases $X_i$ satisfy the following properties:
\begin{itemize}
\item Each block of digits corresponding to an appearance of $[\alpha_i]^{n_i}$ is unique.
\item The digits corresponding to each appearance of $\beta_i$ are in the interval $\{i+1,i+2,\dots,\beta_i /i\}$.
\end{itemize}

To see that $\taursx$ is not in $\DNQ$, we make use of Theorem \ref{thm:distributionnormal}. We note that asymptotically half of the bases $q_n$ are of the form $\beta_i$ for some $i$, and by the previous paragraph, we have that the corresponding digits $E_n$ are $o(q_n)$. Therefore the sequence $(E_n/q_n)_{n=1}^\infty$ is clearly not uniformly distributed modulo $1$.

To show that $\taursx$ is in $\NQ$, we make use of the following lemma, whose proof is elementary.

\begin{lem}\label{lem:limitreduction}
Let $(a_n)_{n=1}^\infty$ and $(b_n)_{n=1}^\infty$ be sequences of positive real numbers such that $\sum_{n=1}^\infty b_n = \infty$. Let $(n_i)_{i=0}^\infty$ be an increasing sequence of positive integers with $n_0=1$ and define $A_m = \sum_{n=n_{m-1}}^{n_m-1} a_n$ and $B_m = \sum_{n=n_{m-1}}^{n_m-1} b_n$. Suppose that
\[
\lim_{m\to \infty} \frac{A_m}{B_m} = 1 \qquad \text{ and } \qquad B_m = o\left( \sum_{i=1}^{m-1} B_i\right),
\]
then
\[
\lim_{n\to \infty} \frac{a_1+a_2+\dots+a_n}{b_1+b_2+\dots+b_n} = 1.
\]
\end{lem}

Let us denote the $j$th appearance of $X_i$ in the bases of $Q$ by $X_{i,j}$. In particular, this will consist of the bases $q_n$ where $n$ falls into the following interval
\[
[N_{i,j},M_{i,j}]:=\left[ \sum_{k=1}^{i-1} 2 L_k \ell_k + 2(j-1)\ell_k+1, \sum_{k=1}^{i-1}2L_k \ell_k +2j \ell_k\right].
\]

Let us write 
\[
Q^{(k)}(X_{i,j}) = \sum_{n=N_{i,j}}^{M_{i,j}} \frac{1}{q_n q_{n+1} q_{n+2}\cdots q_{n+k-1}}
\]
and let $N(B,\taursx,X_{i,j})$ denote the number of occurrences of the block of digits $B$ in the $Q$-Cantor series expansion of $\taursx$ with the first digit of the block occurring at the $n$th place, with $n\in [N_{i,j},M_{i,j}]$.

Comparing these two definitions with the definition of $Q$-normality in \refeq{rnormal}, and using Lemma \ref{lem:limitreduction}, we see that it suffices to show that
\begin{equation}\label{eq:final1}
N(B,\taursx,X_{i,j})= Q^{(k)}(X_{i,j})(1+o(1))
\end{equation}
as $i$ increases (uniformly for any $j\in [1,L_i]$) and that
\begin{equation}\label{eq:final2}
Q^{(k)}(X_{i,j}) = o\left( \sum_{l=1}^{L_{i-1}} Q^{(k)}(X_{i-1,l})\right)
\end{equation}
as $i$ increases.

To estimate the size of $Q^{(k)}(X_{i,j})$, we note that most of the contribution comes from the terms when $q_n=q_{n+1}=\dots=q_{n+k-1}=i$. There are precisely $\ell_i(n_i-k)$ such terms. If any of the $q$'s in the denominator of a term equals $(i!)^2$ (or, possibly $(i+1)!^2$), then the entire term is at most $i^{-k+1}(i!)^{-2}$. And there are precisely $\ell_i(n_i+k)$ such summands. Therefore,
\begin{equation}\label{eq:final3}
Q^{(k)}(X_{i,j}) = \frac{\ell_i(n_i-k)}{i^k} +O\left( \frac{\ell_i(n_i+k)}{i^{k-1} (i!)^2}\right) =  \frac{\ell_in_i}{i^k}(1+o(1))
\end{equation}
where the $o(1)$ is decreasing as $i$ increases and is uniform over $j\in [1,L_i]$. 

From this, we derive
\begin{equation}\label{eq:final4}
\sum_{j=1}^{L_{i-1}} Q^{(k)}(X_{i-1,j}) = L_{i-1} \frac{\ell_{i-1}n_{i-1}}{(i-1)^k}(1+o(1))
\end{equation}
and therefore \eqref{eq:final2} derives from comparing \eqref{eq:final3} and \eqref{eq:final4} and using the definition of $L_{i-1}$.

To estimate the size of $N(B,\taursx,X_{i,j})$, let us suppose that $i$ is sufficiently large so that the digits of $B$ are less than $i$ and so that all the digits of $\taursx$ corresponding to the large bases $(i!)^2$ are at least $i$ in size. Therefore $B$ will only occur in the digit strings corresponding to the small blocks $[i]^{n_i}$. We know that there are $\ell_b$ such distinct digit strings and at most $i^{n_i}e^{-n_i^{1/5}}$ of them can not be $(\epsilon_i,k)$-normal. Therefore, we have
\begin{equation}\label{eq:final5}
N(B,\taursx,X_{i,j}) = \left(i^{-k}+O(\epsilon_i) \right)n_i \ell_i + O\left( n_i i^{n_i}e^{-n_i^{1/5}}\right)= \frac{n_i\ell_i}{i^k} (1+o(1)).
\end{equation}
As before, the $o(1)$ here is decreasing as $i$ increases.

Comparing \eqref{eq:final3} and \eqref{eq:final5} gives \eqref{eq:final1} and completes the proof.

%%%%%%%%%%%%%%%%%%%%%%%%%%%%%%
\section{Proof of Theorem \ref{thm:computability}}
%%%%%%%%%%%%%%%%%%%%%%%%%%%%%%

We shall, in fact, prove the following, more explicit theorem.

\begin{thm}
The basic sequence $Q$ given in \eqref{eq:Qconstruction} is computable.  Let $\eta = 0. E_1 E_2 \cdots$ \wrt ${Q}$ be the real number from the set $\Theta(\alpha, \beta, s, t, \upsilon, F, I)$ given in Section \ref{sec:mainproof} such that $E_n=i_\alpha(n)!$ if \refeq{VNcond} holds (that is, the digits corresponding to the bases $(i!)^2$ will be $i!$).
\end{thm}

\begin{proof}

The sequence $\floor{\log(i)}$ is computable,  so $n_i=i^{\floor{\log(i)}}$ is a computable sequence. 
We can create a Turing machine that, given input $i$, lexicographically enumerates all integers in $[0, i^{n_i}-1]$.
Moreover, we  use two Turing machines that, given input $i$ and the list of integers, check if each integer $\overline{B}$ satisfies the conditions $i! \overline{B} < i^{n_i}$ and $i! | \overline{B}$ since the order relation on integers and divisibility of integers are computable relations. 
We can then create a Turing machine that, given input $i$, lexicographically enumerates the elements of $\mathcal{L}_i$. 
Another Turing machine can be used to output the size of $\mathcal{L}_i$. Thus, $(l_i)$ is a computable sequence. 
Since $(n_i)$ and $(l_i)$ are computable sequences, the sequence $(L_i)$ is also  computable. Furthermore, $( 2 L_i l_i n_i )$ is also a computable sequence.

Thus the sequences $(\alpha_i)$, $(\beta_i)$, $(s_i)$, $(t_i)$, and $(\upsilon_i)$ are all computable sequences. 
Therefore we can create a Turing machine $A$ to output the $n$th term of $Q(\alpha, \beta, s, t, \upsilon)$ as follows. 
First make a Turing machine $B$ that on inputs $i$ and $n$ will output the $n$th base of $X_i$ as follows. 
Determine the residue class of $n$ modulo $2 n_i$. 
If this residue is less than $n_i$, return $i$, otherwise return $(i!)^2$. 
This computes the $n$th digit of $X_i$. 
Finally, create the Turing machine $C$ that on input $n$ determines the maximum $i$ such that $2 L_i l_i n_i < n$ and computes $N = n - \sum_{j=1}^i 2 L_i l_i n_i$.
Then define $A$ as the Turing machine that on input $n$ computes $B(C(n))$. 
Thus, we have a Turing machine the outputs the $n$th base of $Q$, so $Q$ is a computable sequence.

By an argument from the previous paragraphs, we have that there is a Turing machine that on input $i$ lexicographically enumerates $\mathcal{L}_i$. 
We can construct a Turing machine to compute the sequence $(E_n)$ as follows. 
Use the Turing machine $D$ that on input $n$ outputs $(m,N)$ where $m = \min\{ j : \sum_{k=1}^j 2L_i l_i n_i < n \}$ and $N = n - \sum_{j=1}^m 2 L_j l_j n_j$. 
Create a new Turing machine $E$ that on input $i$ and $n$ does the following. 
If the residue class of $n$ modulo $2 n_i$ is greater than or equal to $n_i$, output $i!$. 
Otherwise, compute $z = \floor{n/(2 n_i)}$ and return the $n \mod n_i$th digit of the $z$th element of $\mathcal{L}_i$. 
Then the Turing machine that on input $n$, runs the $D$ on $n$, and then runs $E$ on the output of the $D$, computes the sequence $(E_n)$. 
Since both $(E_n)$ and $(q_n)$ are computable sequences, the real number $\eta = \sum_{n=1}^\infty \frac{E_n}{q_1 \cdots q_n}$ is computable.
\end{proof}

\section{Further problems}

The effect of the rational number $s$ on the set we constructed to prove Theorem \ref{thm:main} was negligible. We specifically constructed $Q$ so that the denominator of $s$ had to divide some $q_n$, so addition by $s$ would never change more than a finite amount of digits by Corollary \ref{lem:rationaladd}, and thus had no impact on either $Q$-normality or $Q$-distribution normality (or the lack thereof).  This suggests the following natural question.
\begin{problem}
If we were to restrict $Q$ so that, say $3\not|~q_n$ for any $n$, then addition by $1/3$ would have to change an infinite number of digits. Are results similar to those given here possible for such $Q$?
\end{problem}
We also ask
\begin{problem}
 Does a version of \reft{main} hold for all $Q$  that are infinite in limit and fully divergent?
\end{problem}

\begin{problem}
There exist some basic sequences $Q$ where the set $\DNQ$ does not contain any computable real numbers.  See \cite{BerosBerosComputable}.  What assumptions on $Q$ must we have to guarantee that there are computable real numbers in $\urset$?
\end{problem}

\begin{problem}
Can a version of \reft{main} or \reft{computability} be stated for normality of order $k$?
\end{problem}

%\begin{thebibliography}{10}

\bibliographystyle{amsplain}

%\bibliography{mance} 

\begin{thebibliography}{10}

\bibitem{AgafonovNormal}
V.~N. Agafonov, \emph{Normal sequences and finite automata}, Dokl. Akad. Nauk
  SSSR \textbf{179} (1968), 255--256.

\bibitem{AireyManceNormalPreserves}
D.~Airey and B.~Mance, \emph{Normality preserving operations for {C}antor
  series expansions and associated fractals part {I}}, arXiv$:$1407.0777.

\bibitem{AireyManceHDDifference}
\bysame, \emph{On the {H}ausdorff dimension of some sets of numbers defined
  through the digits of their {$Q$}-{C}antor series expansions},
  arXiv$:$1407.0776.

\bibitem{AistleitnerNormalPreserves}
C.~Aistleitner, \emph{On modifying normal numbers}, Unif. Distrib. Theory
  \textbf{6} (2011), no.~2, 49--58.

\bibitem{AlMa}
C.~Altomare and B.~Mance, \emph{{C}antor series constructions contrasting two
  notions of normality}, Monatsh. Math \textbf{164} (2011), 1--22.

\bibitem{BecherFigueiraPicchi}
V.~Becher, S.~Figueira, and R.~Picchi, \emph{Turing's unpublished algorithm for
  normal numbers}, Theoret. Comput. Sci. \textbf{377} (2007), no.~1--3,
  126--138.

\bibitem{BerosBerosComputable}
A.~A. Beros and K.~A. Beros, \emph{Normal numbers and limit computable {C}antor
  series}, arXiv$:$1404.2178.

\bibitem{BugeaudBook}
Y.~Bugeaud, \emph{Distribution modulo one and {D}iophantine approximation},
  Cambridge University Press, Cambridge, 2012.

\bibitem{Cantor}
G.~Cantor, \emph{\"{U}ber die einfachen {Z}ahlensysteme}, Zeitschrift f\"{u}r
  Math. und Physik \textbf{14} (1869), 121--128.

\bibitem{ChangNormal}
K.~T. Chang, \emph{A note on normal numbers}, Nanta Math. \textbf{9} (1976),
  70--72.

\bibitem{LutzNormalityPreserves}
D.~Doty, J.~H. Lutz, and S.~Nandakumar, \emph{Finite-state dimension and real
  arithmetic}, Inform. and Comput. \textbf{205} (2007), 1640--1651.

\bibitem{ErdosRenyiConvergent}
P.~Erd\H{o}s and A.~R\'{e}nyi, \emph{On {C}antor's series with convergent $\sum
  1/q_n$}, Annales Universitatis L. E\"{o}tv\"{o}s de Budapest, Sect. Math.
  (1959), 93--109.

\bibitem{ErdosRenyiFurther}
P.~Erd\H{o}s and A.~R\'enyi, \emph{Some further statistical properties of the
  digits in {C}antor's series}, Acta Math. Acad. Sci. Hungar \textbf{10}
  (1959), 21--29.

\bibitem{FurstenbergDisjoint}
H.~Furstenberg, \emph{Disjointness in ergodic theory, minimal sets, and a
  problem in {D}iophantine approximation}, Math. Systems Theory \textbf{1}
  (1967), 1--49.

\bibitem{Galambos}
J.~Galambos, \emph{Representations of real numbers by infinite series}, Lecture
  Notes in Math., vol. 502, Springer-Verlag, Berlin, Hiedelberg, New York,
  1976.

\bibitem{Kamae}
T.~Kamae, \emph{Subsequences of normal sequences}, Israel J. Math \textbf{16}
  (1973), 121--149.

\bibitem{KamaeWeiss}
T.~Kamae and B.~Weiss, \emph{Normal numbers and selection rules}, Israel J.
  Math \textbf{21} (1975), 101--110.

\bibitem{ppq1}
B.~Mance, \emph{Number theoretic applications of a class of {C}antor series
  fractal functions part {I}}, To appear in Acta Math. Hungar. (2014).

\bibitem{Mance4}
\bysame, \emph{Typicality of normal numbers with respect to the {C}antor series
  expansion}, New York J. Math. \textbf{17} (2011), 601--617.

\bibitem{X}
Ch. Mauduit, \emph{Problem session dedicated to {G}\'{e}rard {R}auzy},
  Dynamical systems (Luminy--Marseille, 1998), World Scientific Publishing,
  River Edge, NJ, 2000.

\bibitem{MerkleReimann}
W.~Merkle and J.~Reimann, \emph{Selection functions that do not preserve
  normality}, Theory Comput. Syst. \textbf{39} (2006), no.~5, 685--697.

\bibitem{RenyiProbability}
A.~R\'{e}nyi, \emph{On a new axiomatic theory of probability}, Acta Math. Acad.
  Sci. Hungar. \textbf{6} (1955), 329--332.

\bibitem{Renyi}
\bysame, \emph{On the distribution of the digits in {C}antor's series}, Mat.
  Lapok \textbf{7} (1956), 77--100.

\bibitem{RenyiSurvey}
\bysame, \emph{Probabilistic methods in number theory}, Shuxue Jinzhan
  \textbf{4} (1958), 465--510.

\bibitem{SchweigerCantor}
F.~Schweiger, \emph{\"{U}ber den {S}atz von {B}orel-{R}\'{e}nyi in der
  {T}heorie der {C}antorschen {R}eihen}, Monatsh. Math. \textbf{74} (1969),
  150--153.

\bibitem{Sierpinski}
M.~W. Sierpi\'{n}ski, \emph{D\'{e}monstration \'{e}l\'{e}mentaire du
  th\'{e}or\'{e}m de {M}. {B}orel sur les nombres absolument normaux et
  d\'{e}termination effective d'{u}n tel nombre}, Bull. Soc. Math. France
  \textbf{45} (1917), 125--153.

\bibitem{Turan}
P.~Tur\'{a}n, \emph{On the distribution of ``digits'' in {C}antor systems},
  Mat. Lapok \textbf{7} (1956), 71--76.

\bibitem{Turing}
A.~M. Turing, \emph{{C}ollected {W}orks of {A}. {M}. {T}uring}, North-Holland
  Publishing Co., Amsterdam, 1992.

\bibitem{Salat4}
T.~\u{S}al\'at, \emph{\"{U}ber die {C}antorschen {R}eihen}, Czech. Math. J.
  \textbf{18 (93)} (1968), 25--56.

\bibitem{Wall}
D.~D. Wall, \emph{Normal numbers}, Ph.D. thesis, Univ. of California, Berkeley,
  Berkeley, California, 1949.

\bibitem{Weyl2}
H.~Weyl, \emph{\"{U}ber ein {P}roblem aus dem {G}ebiete der diophantischen
  {A}pproximationen}, Nachr. Ges. Wiss. G\"{o}ttingen, Math.-phys. \textbf{K1}
  (1914), 234--244.

\bibitem{Weyl4}
\bysame, \emph{\"{U}ber die {G}leichverteilung von {Z}ahlen mod. {E}ins}, Math.
  Ann. \textbf{77} (1916), 313--352.

\end{thebibliography}
\providecommand{\bysame}{\leavevmode\hbox to3em{\hrulefill}\thinspace}
\providecommand{\MR}{\relax\ifhmode\unskip\space\fi MR }
\providecommand{\MRhref}[2]{%
  \href{http://www.ams.org/mathscinet-getitem?mr=#1}{#2}
}
\providecommand{\href}[2]{#2}

%\end{thebibliography}

\end{document}